\documentclass[11pt]{amsart}

\makeatletter
\def\@settitle{\begin{center}\baselineskip14\p@\relax\sc\Large\@title\end{center}}
\def\@setauthors{\begingroup\def\thanks{\protect\thanks@warning}\trivlist\centering\footnotesize \@topsep30\p@\relax\advance\@topsep by -\baselineskip\item\relax\author@andify\authors\def\\{\protect\linebreak}\sc\large{\authors}\ifx\@empty\contribs\else,\penalty-3 \space \@setcontribs\@closetoccontribs\fi\endtrivlist\endgroup}
\renewcommand*{\@seccntformat}[1]{\csname the#1\endcsname.\hspace{1mm}} 
\makeatother

\usepackage{amssymb}

\title[Reduction of Segre structures]{Reduction of~$\beta$-integrable\\ 2-Segre structures}
\author[T. Mettler]{Thomas Mettler}
\date{November 13, 2012}
\address{Forschungsinstitut f\"ur Mathematik, ETH Z\"urich, R\"amistrasse 101, CH-8092 Z\"urich, Switzerland}
\email{thomas.mettler@fim.math.ethz.ch}
\thanks{Thanks to Schweizerischer Nationalfonds for its support via the postdoctoral fellowship PBFRP2-133545, to the Mathematical Sciences Research Institute in Berkeley for its support via a postdoctoral fellowship in the period September 2011-June 2012, and to the Forschungsinstitut f\"ur Mathematik at ETH Z\"urich for its support in the period July-December 2012.}

\newcommand{\im}{\operatorname{im}}
\newcommand{\eds}{\textsc{eds}}
\renewcommand{\Re}{\operatorname{Re}}
\renewcommand{\Im}{\operatorname{Im}}
\newcommand{\GL}{\mathrm{GL}}
\renewcommand{\P}{\mathbb{P}}
\newcommand{\V}{V_{\C}}
\newcommand{\f}{\mathbb{CP}^1\setminus \mathbb{RP}^1}

\newcommand{\id}{\mathrm{I}}
\renewcommand{\d}{\mathrm{d}}
\renewcommand{\i}{\mathrm{i}}
\newcommand{\B}{F_\mathcal{S}}
\newcommand{\R}{\mathbb{R}}
\newcommand{\C}{\mathbb{C}}
\newcommand{\AJ}{\mathfrak{J}}
\newcommand{\X}{X_\mathcal{S}}
\renewcommand{\H}{S^1\cdot \,\GL(n,\R)}
\newcommand{\GR}{G^+_2(\R^{n+2})}
\newcommand{\cip}{\mathbb{CP}^{n+1}\setminus\mathbb{RP}^{n+1}}

\newcommand{\E}{\Pi}

\newcommand{\G}{H^+(2,n)}

\newtheorem{theorem}{Theorem}[section]
\newtheorem{lemma}{Lemma}[section]
\newtheorem{corollary}{Corollary}[section]
\newtheorem{proposition}{Proposition}[section]

\theoremstyle{definition}
\newtheorem{defi}{Definition}[section]
\newtheorem{remark}{Remark}[section]
\theoremstyle{remark}
\newtheorem{ex}{Example}[section]

\numberwithin{equation}{section}

\begin{document}

\maketitle

\begin{abstract}
We show that locally every $\beta$-integrable $(2,n)$-Segre structure can be reduced to a torsion-free $\H$-structure. This is done by observing that such reductions correspond to sections with holomorphic image of a certain `twistor bundle'. For the homogeneous $(2,n)$-Segre structure on the oriented~$2$-plane Grassmannian, the reductions are shown to be in one-to-one correspondence with the smooth quadrics~$Q \subset \mathbb{CP}^{n+1}$ without real points.  
\end{abstract}


\section{Introduction}

We study the problem of reducing the $G$-structure associated to a certain type of Segre structure to a torsion-free substructure. 

Segre - or closely related structures and their counterparts in the category of complex manifolds were studied under various names, including tensor product structure~\cite{MR0200858}, generalised conformal structure~\cite{MR925263}, complex paraconformal structure~\cite{MR1085595}, (almost) Grassmann structure~\cite{MR1406793,MR2389257,MR0291976,MR1804138}, Segre structure~\cite{MR1427757,MR1847382}, and in~\cite{MR1106939} as an example of a class of structures called almost symmetric hermitian manifolds.  

Here, by an $(m,n)$-Segre structure on a manifold $M$ we mean a smoothly varying family of cones $\mathcal{S}_p\subset T_pM$ in the tangent spaces of $M$, each linearly isomorphic to the Segre cone of linear maps $\R^m\to \R^n$ of rank one. The tangent planes to $M$ which are contained in some Segre cone $\mathcal{S}_p$ come in two types, called $\alpha$- and $\beta$-planes. An immersed submanifold $\Sigma \subset M$ whose tangent planes are all $\beta$-planes and which is maximal in the sense of inclusion is called a $\beta$-surface. A Segre structure is called $\beta$-integrable, if every $\beta$-plane is tangent to a unique $\beta$-surface.

In~\cite{MR1847382} Grossman showed that the space of paths of a certain class of geodesically simple path geometries, which he calls torsion-free, inherits a Segre structure. Bryant observed in~\cite{MR1898190} that the space of oriented geodesics~$\Lambda$  of a geodesically simple Finsler structure of constant flag curvature (\textsc{cfc}) $1$ inherits a K\"ahler structure and a torsion-free~$\H$-struc\-ture satisfying a certain positivity condition. Conversely, he shows that every torsion-free $\H$-structure satisfying the positivity condition (and an integrability condition for $n=2$) arises via a (generalised) \textsc{cfc} $1$ Finsler structure.

The main result of the article is that locally every $\beta$-integrable $(2,n)$-Segre structure $\mathcal{S}$ can be reduced to a torsion-free~$\H$-struc\-ture. It follows with Bryant's result, that locally every $\beta$-integrable $(2,n)$-Segre structure admitting a $\H$-reduction satisfying the positivity condition of~\cite{MR1898190} arises via a (generalised) \textsc{cfc} $1$ Finsler structure. 

Note that an $\H$-reduction of a $\beta$-integrable $(2,n)$-Segre structure $\mathcal{S}$ equips the underlying manifold with an integrable almost complex structure which preserves $\mathcal{S}$ and for which the $\beta$-surfaces are totally real.  

This article is organised as follows. In \S\ref{def} we review the construction of a `twistor bundle' $\rho : \X\to M$ over a manifold $M$ which is equipped with a $\beta$-integrable $(2,n)$-Segre structure and show in \S\ref{holo} that~$\rho$-sections having holomorphic image are in one-to-one correspondence with reductions of~$\mathcal{S}$ to torsion-free~$\H$-structures on~$M$. It follows that locally every~$\beta$-integrable~$(2,n)$-Segre structure can be reduced to a torsion-free~$\H$ structure. In \S\ref{quad} we show that for the homogeneous $(2,n)$-Segre structure on the oriented~$2$-plane Grassmannian~$M=G^+_2(\R^{n+2})$, the reductions are in one-to-one correspondence with the smooth quadrics~$Q \subset \mathbb{CP}^{n+1}$ without real points.   
 
\begin{remark}
Before this work begun Robert Bryant informed the author about his private notes regarding the generality of positive constant flag curvature (\textsc{cfc}) Finsler structures on the~$n$-sphere. He shows that a positive \textsc{cfc} Finsler structure on the~$n$-sphere all of whose geodesics are closed and of the same length gives rise to a~$D^2$-bundle~$\rho : X \to \Lambda$, fibering over the space of oriented geodesics $\Lambda$, whose total space is a complex manifold. This bundle is isomorphic to~$\rho_0 : \cip \to G^+_2(\R^{n+2})$ in the case of a rectilinear Finsler structure. In addition, the Finsler structure induces a~$\rho$-section having holomorphic image (isomorphic to a quadric in the rectilinear case) and conversely every such section satisfying a certain convexity condition gives rise to a Finsler structure on~$S^n$ sharing the same geodesics. Using Kodaira deformation theory this allows Bryant to determine the generality of such Finsler structures sharing the same unparametrised geodesics. Although being related, the results in this article were arrived at independently. 
\end{remark}

\section{2-Segre structures}\label{def}

\subsection{Definitions and examples}

Let $m,n \geq 2$ be integers. A vector~$v \in \R^m\otimes \R^n$ is called \textit{decomposable} or \textit{simple} if there exists~$x \in \R^m$ and~$y \in \R^n$ such that~$v=x\otimes y$. Let~$X=(X_\alpha)$ be linear coordinates on~$\R^m$ and~$Y=(Y^k)$ on~$\R^n$. Writing~$Z^k_{\alpha}=X_{\alpha}\otimes Y^k$, the set of simple vectors in~$\R^m\otimes \R^n$ is the zero locus of the homogeneous quadratic equations
\begin{equation}\label{defsegre}
Z^i_{\alpha}Z^k_{\beta}-Z^i_{\beta}Z^k_{\alpha}=0
\end{equation}
and thus is a cone. A subset~$\mathcal{S}$ in a real vector space~$V$ is called an~$(m,n)$-\textit{Segre cone}, if there exists an isomorphism~$V \simeq \R^m\otimes \R^n$ which yields a bijection between~$\mathcal{S}$ and the cone of simple vectors in~$\R^m\otimes \R^n$. 

Let $\mathcal{S}\subset V$ be a Segre cone. Clearly, the isomorphism $V \simeq \R^m\otimes\R^n$ is unique up to composition with an element of the group $G(m,n)$, the subgroup of $\mathrm{GL}(\R^m\otimes\R^n)$ consisting of maps preserving the cone of simple vectors. Let $H(m,n)=\mathrm{GL}(m,\R)\otimes \mathrm{GL}(n,\R)$ and $\mathbb{Z}_2\subset G(n,n)$ be the subgroup generated by the involution $x\otimes y \mapsto y \otimes x$ for $x,y \in \R^n$. Then we have an isomorphism of Lie groups\footnote{For a proof, see the appendix.}
\begin{equation}\label{strucgroup}
G(m,n)\simeq\left\{\begin{array}{cc} H(m,n) & n \neq m,\\ H(n,n)\rtimes \mathbb{Z}_2  & n=m.\\ \end{array}\right.
\end{equation}
\begin{defi}
An~$(m,n)$-\textit{Segre structure}~$\mathcal{S}$ on a smooth~$mn$-manifold $M$ is a choice of an~$(m,n)$-Segre cone~$\mathcal{S}_p\subset T_pM$ in each tangent space of~$M$ which varies smoothly from point to point. 
\end{defi}
An isomorphism $f : T_pM \to \R^m\otimes\R^n$ will be called a \textit{Segre coframe} at $p$ if it maps $\mathcal{S}_p$ to the cone of simple vectors in $\R^m\otimes\R^n$. 

The set of Segre coframes at $p$ will be denoted by $(\B)_p$ and is the fibre of a principal right $G(m,n)$-bundle $\pi : \B \to M$, with right action given by $R_g(f)=g^{-1}\circ f$ for $g \in G(m,n)$. The tautological $1$-form $\zeta$ is defined by requiring that $\zeta_f=f\circ \pi^{\prime}_f : T_f\B \to \R^m\otimes\R^n$ for $f \in \B$. It satisfies $R_g^*\zeta=g^{-1}\circ \zeta$ for $g \in G(m,n)$.

A linear subspace~$\E\subset\mathcal{S}_p$ is called \textit{simple}. The simple linear subspaces which are of the form $\Pi \simeq \R^m\otimes y$ for some $y \in \R^n$ are called $\alpha$-planes and the simple linear subspaces which are of the form $\Pi \simeq x \otimes \R^n$ for some $x \in \R^m$ are called $\beta$-planes.\footnote{The reader is warned that often the opposite convention regarding $\alpha$- and $\beta$-planes is used. The convention used here is chosen to be consistent with~\cite{MR2286630}.} 
Note that $\alpha$- and $\beta$-planes are not well defined for $n=m$ unless the Segre structure has been reduced to an $H(n,n)$-structure. An immersed connected manifold $\Sigma \to M$ for which~$T_p\Sigma$ is a~$\beta$-plane for every point~$p \in \Sigma$ is called a \textit{proto}~$\beta$\textit{-surface}. 
If, in addition,~$\Sigma \to M$ is maximal in the sense of inclusion, then~$\Sigma$ is called a~$\beta$\textit{-surface}. A Segre structure~$\mathcal{S}$ is called~$\beta$\textit{-integrable} if every~$\beta$-plane~$\E$ is tangent to a unique~$\beta$-surface~$\Sigma \to M$. The notion of a (proto)~$\alpha$-surface and~$\alpha$-integrability are defined analogously. The necessary and sufficient conditions for a Segre structure of type~$(m,n)$ to be~$\alpha$- or~$\beta$-integrable were given in~\cite{MR1692442,MR1804138} (see also~\cite{MR1085595} for the complex case).
\begin{ex} Recall that a pseudo-Riemannian metric $g$ on a smooth $4$-manifold $M$ with signature $(+,+,-,-)$ is said to have split-signature. Locally $g$ may be written as 
\begin{equation}\label{nullcone}
g=\varepsilon^1_1\odot \varepsilon^2_2-\varepsilon^1_2 \odot \varepsilon^2_1
\end{equation}
for some linearly independent $1$-forms $\varepsilon^i_j$. A vector~$v \in TM$ is called \textit{null} if~$g(v,v)=0$. It follows with \eqref{nullcone} that the $g$-null vectors give rise to a~$(2,\! 2)$-Segre structure on $M$ and conversely it can be shown that every $(2,\! 2)$-Segre structure on a $4$-manifold $M$ gives rise to a unique conformal structure of split-signature on~$M$.
\end{ex}
Closely related to Segre structures is the notion of an almost Grassmann structure. 
\begin{defi}
A smooth $mn$-manifold $M$ is said to carry an \textit{almost Grassmann structure} if there exist smooth vector bundles $E_m\to M$ and $E_n \to M$ of rank $m,n$ respectively together with an isomorphism $TM \simeq E_m\otimes E_n$.
\end{defi}
\begin{remark}
Clearly, an almost Grassmann structure on $M$ induces a unique Segre structure on $M$, but the existence of a Segre structure $\mathcal{S}$ on $M$ is in general not sufficient for the existence of an almost Grassmann structure inducing $\mathcal{S}$. The two definitions are however equivalent when $m+n$ is odd (see the appendix).
\end{remark}
\begin{ex}The prototypical example of a manifold carrying an almost Grassmann structure is the Grassmannian of $m$-planes in $\R^{m+n}$ (see for instance~\cite{MR1804138} for details). We will construct the associated Segre structure in \S4 for the case $m=2$.
\end{ex}
\subsection{The structure equations of a 2-Segre structure}

We will henceforth restrict our attention to $H(2,n)$-structures $\pi : \B \to M$ and simply speak of $2$-Segre structures, thus implicitly assuming that $(2,2)$-Segre structures have been reduced to $H(2,2)$-structures. We think of $H(2,n)$ as a subgroup of $\mathrm{GL}(2n,\R)$ via the Kronecker product and consequently of $\pi : \B \to M$ as a reduction with structure group $H(2,n)$ of the full coframe bundle $F\to M$ whose fibre at $p \in M$ consists of all isomorphisms $T_pM \to \R^{2n}$. 

A linear connection $\theta$ on $F \to M$ is said to be \textit{adapted} to the $2$-Segre structure $\pi : \B \to M$ if $\theta$ pulls-back to $\B$ to become a principal $H(2,n)$-connection. A $2$-Segre structure is called \textit{torsion-free}, if it admits an adapted connection $\theta$ with vanishing torsion $\tau=\d \zeta + \theta \wedge \zeta$.

Write $H=H(2,n)$ and $\mathfrak{h}\subset \mathfrak{gl}(2,\R)\otimes \mathfrak{gl}(n,\R)$ for the Lie algebra of $H$. For computational purposes it is convenient to introduce the matrices 
$$
a=\left(\begin{array}{rr} 0&-1\\1 & 0\end{array}\right), \quad b_1=\left(\begin{array}{rr} 0&1\\0 & 0\end{array}\right), \quad b_2=\left(\begin{array}{rr} 0&0\\0 & 1\end{array}\right),
$$
and write~$e^i_k$ for the~$(n\times n)$-matrix whose entry is~$1$ at the position~$(k,i)$ and~$0$ otherwise. Using this notation an~$\mathfrak{h}$-basis is given by 
$$
a\otimes \id_n, \quad b_1\otimes \id_n, \quad b_2 \otimes \id_n, \quad \id_2 \otimes e^i_k, 
$$
and a principal $H$-connection $\theta$ on $\B$ may be written as 
\begin{equation}\label{defcon2}
\theta=\chi\otimes\id_n+\id_2\otimes \phi
\end{equation}
with~$\chi=\omega a + 2\xi_1b_1+2\xi_2b_2$ and~$\phi=\phi^i_ke^k_i$ for some linearly independent~$1$-forms~$\omega,\xi_1,\xi_2,\phi^i_k$ on~$\B$. Let $\xi=\xi_1+\i\xi_2$. Straightforward computations show that we may linearly identify $\R^{2n}$ with $\C^n$ in such a way that we can write the first structure equation in complex form: 
\begin{proposition}\label{strucprop}
The connection form~$(\omega,\xi,\phi)$ of an $\B$-adapted con\-nection $\theta$ satisfies 
\begin{equation}\label{struc}
\d\zeta=-\big(\i\left(\omega-\xi\right)\id_n+\phi\big)\wedge\zeta-\i\,\xi\,\id_n \wedge \bar{\zeta}+\tau. 
\end{equation}
\end{proposition}
\begin{proof}
Omitted. 
\end{proof}
Here the forms $\zeta$ and $\tau$ are thought to be $\C^n$-valued and~$\bar{\zeta}$ denotes the~$\C^n$-valued~$1$-form on~$\B$ which is obtained by complex conjugation of the entries of~$\zeta$. 

We have the curvature forms
\begin{equation}\label{curvatureforms}
\aligned
\Omega&=\d\omega+\omega\wedge\i\,(\xi-\bar{\xi}),\\
\Xi&=\d\xi+\xi\wedge\i\left(\bar{\xi}-2\omega\right),\\
\Phi&=\d\phi+\phi\wedge\phi-\omega\wedge\left(\xi+\bar{\xi}\right)\id_n.
\endaligned
\end{equation}
Differentiating the structure equation \eqref{struc} gives the \textit{Bianchi-identity}
\begin{equation}\label{bianchi}
\d\tau=\big(\i\left(\Omega-\Xi\right)\id_n+\Phi\big)\wedge\zeta+\i\,\Xi\,\id_n\wedge\bar{\zeta}. 
\end{equation}
\subsection{A quasiholomorphic fibre bundle} Let $P=\H\subset H(2,n)$ be the closed subgroup consisting of elements of the form
$$
e^{\i\varphi}\cdot b=\left(\begin{array}{rr}\cos \varphi & -\sin \varphi \\ \sin \varphi & \cos \varphi \end{array}\right)\otimes b
$$
for some real number $\varphi$ and $b\in \mathrm{GL}(n,\R)$. Equip the quotient~$\X=\B/P$ with its canonical smooth structure so that the quotient projection~$\nu : \B \to \X$ is a smooth surjective submersion. Note that the~$1$-forms~$\eta^k$ are~$\nu$-semibasic since they are~$\pi$-semibasic. Moreover since~$\theta=(\omega,\xi,\phi)$ is a principal~$H$-connection, it follows with \eqref{defcon2} that~$\xi$ is~$\nu$-semibasic as well. Therefore the forms~$\eta^k$ together with~$\xi_1$ and~$\xi_2$ span the~$\nu$-semibasic~$1$-forms on~$\B$.

\begin{lemma}\label{charcomplex}
Let~$\pi : \B \to M^{2n}$ be a~$2$-Segre structure and~$\theta=(\omega,\xi,\phi)$ an adapted connection.  Then there exists a unique almost complex structure~$\mathfrak{J}$ on~$\X$ such that a complex-valued~$1$-form on~$\X$ is a~$(1,\! 0)$-form for~$\mathfrak{J}$ if and only if its~$\nu$-pullback is a linear combination of~$\left\{\zeta^1,\ldots,\zeta^n,\xi\right\}$ with coefficients in~$C^{\infty}(\B,\C)$. 
\end{lemma}
\begin{proof} Denote by 
$
\frac{\partial}{\partial\eta^l},\; \frac{\partial}{\partial \xi_1}, \; \frac{\partial}{\partial \xi_2}, \; \frac{\partial}{\partial \omega}, \; \frac{\partial}{\partial \phi^i_k}, 
$
the vector fields dual to the coframing~$(\eta^l,\xi_1,\xi_2,\omega,\phi^i_k)$. Define the map~$\tilde{\mathfrak{J}} : T\B \to T\X$ by  
$$
\tilde{\mathfrak{J}}(v)=\nu^{\prime}\left(-\eta^{2k}(v)\frac{\partial}{\partial \eta^{2k-1}}+\eta^{2k-1}(v)\frac{\partial}{\partial \eta^{2k}}-\xi_2(v)\frac{\partial}{\partial \xi_1}+\xi_1(v)\frac{\partial}{\partial \xi_2}\right).
$$
The~$1$-forms~$\eta^l,\xi_1,\xi_2$ are~$\nu$-semibasic and thus we have~$\tilde{\mathfrak{J}}(v+w)=\tilde{\AJ}(v)$ for every~$v \in T\B$ and~$w \in T\B$ tangent to the~$\nu$-fibres. Since~$\theta$ is a principal~$H$-connection, the equivariance~$(R_h)^*\theta=h^{-1}\theta h$ for~$h \in H$ together with a short computation gives 
\begin{equation}\label{rightactioncon}
\left(R_{e^{\i \varphi} \cdot b}\right)^*\xi=e^{-2\i \varphi} \xi, 
\end{equation}
for~$e^{\i \varphi}\cdot b \in S^1\cdot \GL(n,\R)$. Moreover we have 
\begin{equation}\label{rightactioncan}
(R_h)^*\eta=h^{-1}\eta
\end{equation}
for every~$h \in H$. Identifying $\mathrm{GL}(n,\C)$ with the subgroup of $\mathrm{GL}(2n,\R)$ commuting with $a \otimes \mathrm{I}_n$ and using the fact that~$S^1\cdot \GL(n,\R) \subset \GL(n,\C)$ together with (\ref{rightactioncon}, \ref{rightactioncan}) implies
$
\tilde{\mathfrak{J}}\circ\left (R_{e^{\i \alpha} \cdot b}\right)^{\prime}=\tilde{\mathfrak{J}}. 
$
In other words there exists an almost  complex structure~$\mathfrak{J} : T\X \to T\X$ such that~$\tilde{\mathfrak{J}}= \mathfrak{J} \circ \nu^{\prime}$.  Clearly~$\mathfrak{J}$ has the desired properties and these properties uniquely characterise~$\mathfrak{J}$. 
\end{proof}
It is natural to ask when two~$\B$-adapted connections induce the same almost complex structure. We have: 
\begin{lemma}\label{samecomplex}
The~$\B$-adapted connections~$\theta=(\phi,\omega,\xi)$ and~$\theta^{\prime}=(\phi^{\prime},\omega^{\prime},\xi^{\prime})$ induce the same almost complex structure on~$\X$ if and only if~$\xi - \xi^{\prime}=\lambda_k\zeta^k$ for some smooth functions~$\lambda_k : \B \to \C$. In particular any two~$\B$-adapted connections with the same torsion induce the same almost complex structure. 
\end{lemma}
\begin{proof}
Let~$\mathfrak{J}_{\theta}$,~$\mathfrak{J}_{\theta^{\prime}}$ denote the almost complex structures with respect to the connections~$\theta$,~$\theta^{\prime}$ and suppose~$\xi^{\prime}=\xi+\lambda_k\zeta^k$ for some smooth functions~$\lambda_k : \B \to \C$. Let~$\alpha$ be a~$(1,\!0)$-form for~$\mathfrak{J}_{\theta}$. Then we may write
$$
\nu^*\alpha=a_k \zeta^k+a\xi=a_k\zeta^k+a\left(\xi^{\prime}-\lambda_k\zeta^k\right)=\left(a_k-\lambda_k\right)\zeta^k+a\xi^{\prime}
$$
for some smooth functions~$a,a_k : \B \to \C$, thus showing that~$\alpha$ is a~$(1,\! 0)$-form for~$\mathfrak{J}_{\theta^{\prime}}$. Conversely suppose~$\mathfrak{J}_{\theta}=\mathfrak{J}_{\theta^{\prime}}$. Note that~$\xi-\xi^{\prime}$ is~$\pi$-semibasic and may thus be written as 
$$
\xi-\xi^{\prime}=\lambda_k\zeta^k+\lambda_k^{\prime}\bar{\zeta}^k
$$
for some smooth functions~$\lambda_k,\lambda_k^{\prime}: \B \to \C$. Let~$\alpha$ be a~$(1,\!0)$-form for~$\mathfrak{J}_{\theta}$. Then we may write
$$
\nu^*\alpha=a_k\zeta^k+a\xi=a_k^{\prime}\zeta^k+a^{\prime}\xi^{\prime}=a_k^{\prime}\zeta^k+a^{\prime}\left(\xi-\lambda_k\zeta^k-\lambda_k^{\prime}\bar{\zeta}^k\right)
$$
for some smooth functions~$a,a^{\prime},a_k,a_k^{\prime} : \B \to \C$. Thus it follows
$$
(a-a^{\prime})\xi+\left(a_k-a_k^{\prime}+a^{\prime}\lambda_k\right)\zeta^k+a^{\prime}\lambda_k^{\prime}\bar{\zeta}^k=0
$$
which can hold for an arbitrary~$(1,\!0)$-form~$\alpha$ if and only if~$\lambda_k^{\prime}=0$. Finally it is easy to check that if~$(\omega,\xi,\phi)$ and~$(\omega^{\prime},\xi^{\prime},\phi^{\prime})$ are two~$\B$-adapted connections with the same torsion, then there exist~$n$ smooth complex-valued functions~$a_k$ on~$\B$ such that 
\begin{equation}\label{adtors}
\aligned
\omega^{\prime}-\omega&=\Re(a_k)\Im(\zeta^k),\\ \xi^{\prime}-\xi&=\frac{1}{2\i}\bar{a}_k\zeta^k,\\ \left(\phi^{\prime}\right)^i_k-\phi^i_k&=\Re(a_k \zeta^i )+\delta^i_k\Re(a_l)\Re(\zeta^l).\\
\endaligned
\end{equation}
\end{proof}
Denote by~$\mathcal{A}^{1,0}_{\mathcal{S}}$ and~$\mathcal{A}^{0,1}_{\mathcal{S}}$ the complex-valued~$\pi$-semibasic~$1$-forms on~$\B$ which can be written as~$a_k \zeta^k$ and~$a_k \bar{\zeta}^k$ respectively.  Here~$a_k$ are smooth complex-valued functions on~$\B$. Furthermore set
$$
\mathcal{A}^{p,q}_{\mathcal{S}}=\Lambda^p\left(\mathcal{A}^{1,0}_{\mathcal{S}}\right)\otimes \Lambda^q \left(\mathcal{A}^{0,1}_{\mathcal{S}}\right), 
$$
so that the complex-valued~$\pi$-semibasic~$k$-forms~$\mathcal{A}^k_{\mathcal{S}}$ on~$\B$ decompose as
$$
\mathcal{A}^k_{\mathcal{S}}=\bigoplus_{p+q=k}\mathcal{A}^{p,q}_{\mathcal{S}}. 
$$
\begin{proposition}\label{inte}
The almost complex structure~$\mathfrak{J}$ is integrable if and only if~$\Xi$ and the torsion components~$\tau^i$ lie in~$\mathcal{A}^{2,0}_{\mathcal{S}}\oplus \mathcal{A}^{1,1}_{\mathcal{S}}$. In particular for~$n \geq 3$ every torsion-free~$\B$-adapted connection gives rise to an integrable~$\mathfrak{J}$. 
\end{proposition}
\begin{remark}
The integrability conditions for the almost complex structure~$\AJ$ can also be obtained by applying~\cite[Theorem 4]{MR812312}. We we will instead use Proposition \ref{strucprop}. 
\end{remark}
\begin{proof}[Proof of Proposition \ref{inte}]
Using the characterisation of~$\AJ$ provided in Lemma \ref{charcomplex}, the first statement is an immediate consequence of the structure equation \eqref{struc}, the definition of the curvature form~$\Xi$ in \eqref{curvatureforms}, and the Newlander-Nirenberg theorem. In order to prove the second statement we need to show that for~$n\geq 3$ the condition~$\tau=0$ implies~$\Xi \in \mathcal{A}^{2,0}_{\mathcal{S}}\oplus \mathcal{A}^{1,1}_{\mathcal{S}}$.  Since the curvature form~$\Xi$ is a~$\pi$-semibasic~$2$-form we may write
\begin{equation}\label{expxi}
\Xi=x_{kl}\zeta^k\wedge\zeta^l+\tilde{x}_{kl}\bar{\zeta}^k\wedge\zeta^l+\hat{x}_{kl}\bar{\zeta}^k\wedge\bar{\zeta}^l
\end{equation}
for some smooth complex-valued functions~$x_{kl},\tilde{x}_{kl},\hat{x}_{kl}$ on~$\B$. Writing out the Bian\-chi-identity \eqref{bianchi} in components for~$\tau=0$ gives 
$$
0=(\i(\Omega-\Xi)\delta^i_k+\Phi^i_k)\wedge\zeta^k+\i\,\Xi\,\wedge\bar{\zeta}^i,
$$
replacing~$\Xi$ with the expansion \eqref{expxi} we get
$$
0=\cdots+\i\hat{x}_{kl}\bar{\zeta}^k\wedge\bar{\zeta}^l\wedge\bar{\zeta}^i
$$
where the unwritten summands do not contain forms in~$\mathcal{A}^{0,3}_{\mathcal{S}}$. If~$n\geq 3$ there is for every choice of indices~$k,l$ an index~$i \neq k$,~$i \neq l$ so that the Bianchi-identity can hold if and only if~$\hat{x}_{kl}=0$ which is equivalent to~$\Xi$ lying in~$\mathcal{A}^{2,0}_{\mathcal{S}}\oplus \mathcal{A}^{1,1}_{\mathcal{S}}$. 
\end{proof}

\begin{remark}
Recall that~$H(2,2)$-structures $\pi : \B \to M$ correspond to oriented conformal structures of split-signature and thus are always torsion-free. In fact, the logical value of the curvature condition~$\Xi \in \mathcal{A}^{2,0}_{\mathcal{S}}\oplus \mathcal{A}^{1,1}_{\mathcal{S}}$ does not depend on the choice of a particular adapted torsion-free connection, but only on~$\B$. We leave it to the reader to check that this curvature condition corresponds to self-duality\footnote{As in the case of $(4,0)$-signature, a split-signature metric is called self-dual if its Weyl curvature tensor, considered as a bundle-valued 2-form, is its own Hodge-star.} of the associated oriented conformal~$4$-manifold of split-signature. 

In fact, it can be shown that for~$n=2$ the almost complex structure~$\AJ$ is integrable if and only if~$\theta$ is torsion-free and the associated split-signature conformal structure is self-dual. For~$n\geq 3$, the almost complex structure~$\AJ$ is integrable if and only if~$\theta$ is torsion-free.  
\end{remark}
Suppose~$\AJ$ is integrable, so that the total space of the bundle~$\rho : \X \to M$ is a complex~$(n+1)$-manifold. By construction the~$\rho$-fibres are smoothly embedded submanifolds of~$\X$ diffeomorphic to~$\GL(2,\R)/\mathrm{GL}(1,\C)$.  We will argue next, that~$(\X,\AJ)$ is a quasiholomorphic fibre bundle with fibre $\mathbb{CP}^1\setminus \mathbb{RP}^1$. 
\begin{defi}
Let~$\pi : B \to M$ be a fibre bundle with fibre $F$ and~$\AJ~$ an almost complex structure on~$B$. Then~$(B,\AJ)$ is  called \textit{quasiholomorphic} if
\begin{itemize}
\item[(i)]
the almost complex structure $\AJ$ is integrable,
\item[(ii)] there exists a complex structure $I$ on $F$, 
\item[(iii)] each $\pi$-fibre $B_p=\pi^{-1}(p)$ admits a complex structure with respect to which it is biholomorphic to $(F,I)$ and with respect to which the inclusion $B_p \hookrightarrow B$ is a holomorphic embedding.  
\end{itemize}
\end{defi}
Pulling back~$\xi$ with a local section~$s$ of~$\nu : \B \to \X$ gives a complex-valued $1$-form which pulls back to the $\rho$-fibres to be non-vanishing and which depends on~$s$ only up to complex multiples. It follows that the fibres of~$\rho : \X \to M$ are holomorphically embedded Riemann surfaces with respect to the complex structure induced by~$\xi$. Using the Maurer-Cartan form of $\mathrm{GL}(2,\R)$, it is easy to see that fibres are biholomorphic to $\mathbb{CP}^1\setminus\mathbb{RP}^1$. 
In~\cite{MR1692442, MR1804138} it was shown that for~$n \geq 3$ a~$2$-Segre structure is~$\beta$-integrable if and only if it is torsion-free and for~$n=2$ if and only if it is self-dual. Summarising we have:
\begin{theorem}\label{upperholom}
Let~$\pi : \B \to M$ be a~$\beta$-integrable~$2$-Segre structure. Then there exists a canonical almost complex structure~$\AJ$ on~$\X$ so that~$(\X,\AJ)$ is a quasiholomorphic fibre bundle with fibre~$\f$. 
\end{theorem}
\begin{proof}
We pick an~$\B$-adapted connection without torsion and let~$\AJ$ be the associated almost complex structure on~$\X$ whose existence is guaranteed by Lemma \ref{charcomplex}. Then by Proposition \ref{inte} and the above remarks, the almost complex structure~$\AJ$ is integrable and~$(\X,\AJ)$ is a quasiholomorphic fibre bundle with fibre~$\f$. Finally, by Lemma \ref{samecomplex}, any other~$\B$-adapted torsion-free connection gives rise to the same almost complex structure~$\AJ$.  
\end{proof}
\section{Reductions of~$\beta$-integrable~2-Segre structures}\label{holo}
We will henceforth consider the~$\beta$-integrable case and assume~$\rho : \X \to M$ to be equipped with its canonical integrable almost  complex structure~$\AJ$ with respect to which~$(\X,\AJ)$ is a quasiholomorphic fibre bundle. By construction the sections of the bundle~$\rho : \X \to M$~correspond to reductions of the principal~$H$-bundle~$\pi : \B \to M$~with structure group~$\H$. Note that an $\H$-reduction of a $\beta$-integrable $(2,n)$-Segre structure $\pi : \B \to M$ equips $M$ with an almost complex structure preserving the Segre cones $\mathcal{S}_p$ and for which the $\beta$-planes are totally real. In this section we will show that the torsion-free~$\H$-reductions of~$\B$ are in one-to-one correspondence with the sections~$\sigma : M \to \X$ having holomorphic image~$\sigma(M)\subset\X$. This will done using exterior differential systems (\eds). The notation and terminology for \eds~are chosen to be consistent with~\cite{MR1083148}. 

A basis for the Lie algebra of~$\H$ is given by 
$$
a\otimes \id_n, \quad \id_2\otimes e^i_k.
$$
Suppose~$R \to M$ is a torsion-free~$\H$-structure with adapted connection~$\theta$. Write 
$$
\theta=a\alpha\otimes \id_n+\id_2\otimes \beta,
$$
for some~$1$-form~$\alpha$ and some~$\mathfrak{gl}(n,\R)$-valued~$1$-form~$\beta$ on~$R$. Let~$\zeta$ denote the pullback of the canonical~$\C^n$-valued~$1$-form to~$R$, then~$\zeta$ satisfies 
\begin{equation}\label{strucred}
\d\zeta=-\left(\i \alpha\, \id_n+\beta\right) \wedge \zeta,
\end{equation}
as was already observed in~\cite{MR1898190}. 

We will need the following Lemma whose proof is straightforward and thus omitted: 
\begin{lemma}\label{compsub}
Let~$(X,J)$ be a complex~$(n+1)$-manifold,~$(\mu^1,\ldots,\mu^n,\kappa) \in \mathcal{A}^1(X,\C)$ a basis for the~$(1,\!0)$-forms of~$J$ and~$f : \Sigma \to X$ a~$2n$-submanifold with
\begin{equation}\label{nondeg}
f^*\left(\i\mu^1\wedge\bar{\mu}^1\wedge\cdots\wedge\i\mu^n\wedge\bar{\mu}^n\right)\neq 0. 
\end{equation}
Then~$(f,\Sigma)$ is a complex submanifold if and only if 
$$
f^*\left(\kappa\wedge\mu^1\wedge\cdots\wedge\mu^n\right)=0.
$$ Moreover through every point~$p \in X$ passes such a complex submanifold.  
\end{lemma}
On~$\B$ define the exterior differential system 
$$
\mathcal{I}=\langle \xi\wedge\zeta^1\wedge\cdots\wedge\zeta^n\rangle
$$ 
together with the independence condition
$$
Z=\i\zeta^1\wedge\bar{\zeta}^1\wedge\cdots\wedge\i\zeta^n\wedge\bar{\zeta}^n.
$$
The \eds~$(\mathcal{I},Z)$ is of interest because of the following:
\begin{lemma}\label{lemma1}
Let~$\sigma : M \to \X$ be an~$\H$-reduction of~$\B$ and~$\tilde{\sigma} : U \to \B$ a local coframing covering~$\sigma$. Then~$\tilde{\sigma}$ is an integral manifold of~$(\mathcal{I},Z)$ if and only if~$\sigma\vert_U : U \to \X$ is a complex submanifold.   
\end{lemma}
\begin{proof}
Let~$s : \rho^{-1}(U) \to \B$ be a local section of the bundle~$\nu : \B \to \X$ and let~$\mu^i=s^*\zeta^i$ for~$i=1,\ldots,n$ and~$\kappa=s^*\xi$ be a local basis for the~$(1,\!0)$-forms on~$\rho^{-1}(U)$. Then
$$
\aligned
\nu^*\mu^i&=(s\circ \nu)^*\zeta^i=(R_t)^*\zeta^i,\\
\nu^*\kappa&=(s\circ \nu)^*\xi=(R_t)^*\xi, 
\endaligned
$$
for some smooth function~$t : \pi^{-1}(U) \to \H$. Recall that for~$e^{i\varphi}\cdot b \in \H$ we have
$$
\aligned
\big(R_{e^{\i\varphi}\cdot b}\big)^*\xi&=e^{-2\i\varphi}\xi,\\
\big(R_{e^{\i\varphi}\cdot b}\big)^*\zeta&=\left(e^{-\i\varphi} \cdot b^{-1}\right)\zeta.
\endaligned
$$ 
This yields
$$
\nu^*\left(\i\mu^1\wedge\bar{\mu}^1\wedge\cdots\wedge \i\mu^n\wedge\bar{\mu}^n\right)=(\det b)^{-2}\,Z\neq 0
$$
for some smooth map~$b : \pi^{-1}(U) \to \GL(n,\R)$ and
$$
\nu^*\kappa=e^{-2\i\varphi}\xi
$$ 
for some smooth function~$\varphi : \pi^{-1}(U) \to \R$. Hence we get 
$$
\left(\sigma\vert_U\right)^*\left(\i\mu^1\wedge\bar{\mu}^1\wedge\cdots\wedge\i\mu^n\wedge\bar{\mu}^n\right)=\left((\det b)^{-2} \circ \tilde{\sigma}\right)\tilde{\sigma}^*Z
$$
which vanishes nowhere since~$\tilde{\sigma}$ is a~$\pi$-section. Therefore according to Lemma \ref{compsub},~$\sigma\vert_U : U \to \X$ is a complex submanifold if and only if 
$$
\left(\sigma\vert_U\right)^*\left(\kappa\wedge\mu^1\wedge\cdots\wedge\mu^n\right)=\left(\left(\frac{e^{-(n+2)\i\varphi}}{\det b}\right)\circ \tilde{\sigma}\right)\tilde{\sigma}^*\left(\xi\wedge\zeta^1\wedge\cdots\wedge\zeta^n\right)=0. 
$$
\end{proof}
Recall that~$\H \subset \GL(n,\C)$ and we can thus look for reductions~$\sigma : M \to \X$ whose associated almost complex structure~$\AJ_{\sigma}$ is integrable. 
\begin{proposition}\label{prop1}
Let~$\sigma : M \to \X$ be an~$\H$-reduction of~$\pi : \B \to M$. Then the following two statements are equivalent:
\begin{itemize}
\item[(i)] The almost complex structure~$\AJ_{\sigma}$ is integrable. 
\item[(ii)] Any local coframing~$\tilde{\sigma} : U \to \B$ covering~$\sigma$ is an integral manifold of~$(\mathcal{I},Z)$. 
\end{itemize}
\end{proposition}
\begin{proof} Since~$\tilde{\sigma}$ is a~$\pi$-section we have~$\tilde{\sigma}^*Z \neq 0$. Write~$\chi^i=\tilde{\sigma}^*\zeta^i$. The local coframing~$\tilde{\sigma}$ is adapted to the~$\H$-reduction~$\sigma$ and thus the forms~$\chi^i$ are a local basis of the~$(1,\!0)$-forms of~$\AJ_{\sigma}$. By Newlander-Nirenberg~$\AJ_{\sigma}$ is integrable if and only if there exist complex-valued~$1$-forms~$\pi^i_k$ such that
$$
d\chi^i=\pi^i_k\wedge\chi^k. 
$$
Using the structure equation \eqref{struc} we get 
\begin{equation}\label{step}
\d\chi^i=\tilde{\sigma}^*d\zeta^i=\tilde{\pi}^i_k\wedge\chi^k-\i\tilde{\sigma}^*\xi\wedge\bar{\chi}^i
\end{equation}
for some complex-valued~$1$-forms~$\tilde{\pi}^i_k$. Write 
$$
\tilde{\sigma}^*\xi=x_k\chi^k+y_k\bar{\chi}^k 
$$
for some smooth complex-valued functions~$x_k,y_k$ on~$U$. Then \eqref{step} implies that~$\AJ_{\sigma}$ is integrable on~$U$ if and only if the functions~$y_k$ all vanish. This condition is equivalent to~$\tilde{\sigma}$ being an integral manifold of~$(\mathcal{I},Z)$.  
\end{proof}

We are now ready to prove: 
\begin{theorem}\label{main}
Let~$\pi : \B \to M$ be a~$\beta$-integrable~$2$-Segre structure. Then an~$\H$-reduction~$R \subset \B$ is torsion-free if and only if~$\nu(R)\subset \X$ is a complex submanifold.  
\end{theorem}
\begin{proof}
Let~$\nu(R)=\sigma(M)$ for some~$\rho$-section~$\sigma : M \to \X$ which has holomorphic image, then by Lemma \ref{lemma1} and Proposition \ref{prop1}, the almost complex structure~$\AJ_{\sigma}$ is integrable. This is equivalent to~$\xi$ satisfying~$\xi=x_k \zeta^k$ for some smooth complex-valued functions~$x_k$ on~$R$. Pulling back the structure equation \eqref{struc} to~$R\subset \B$ gives  
\begin{equation}\label{pullbackr}
\d\zeta=-\big(\i\left(\omega-x_k \zeta^k\right)\id_n+\phi\big)\wedge\zeta-\i\,x_k \zeta^k\,\id_n \wedge \bar{\zeta}.
\end{equation}
Define 
$$
\aligned
\alpha&=\omega-\Im(x_k)\Im(\zeta^k),\\
\beta^j_l&=\phi^j_l-\Re(\i \bar{x}_l\zeta^j)-\delta^j_l\Im(x_k)\Re(\zeta^k),\\
\endaligned
$$
then the~$1$-form~$\theta=a\alpha\otimes\id_n+\id_2\otimes\beta$ is a linear connection on~$R$ which satisfies
\begin{equation}\label{struc2}
\d\zeta=-\left(\i\alpha\id_n+\beta\right)\wedge\zeta,
\end{equation}
thus~$R$ is torsion-free. Conversely suppose the reduction~$\sigma : M \to \X$ is torsion-free, so that on~$R=(\nu^{-1}\circ \sigma)(M)$ there exists a linear connection~$\theta=a\alpha\otimes\id_n+\id_2\otimes\beta$ satisfying \eqref{struc2}. Pulling back~$(\omega,\xi,\phi)$ to~$R$ gives
\begin{equation}\label{pullconn2}
\aligned
\omega&=\alpha+a_k\left(\zeta^k+\bar{\zeta}^k\right)+\i\tilde{a}_k\left(\zeta^k-\bar{\zeta}^k\right)\\
\xi&=x_k\zeta^k+\tilde{x}_{k}\bar{\zeta}^k\\
\phi^i_k&=\beta^i_k+f^i_{kl}\left(\zeta^k+\bar{\zeta}^k\right)+\i \bar{f}^i_{kl}\left(\zeta^k-\bar{\zeta}^k\right)
\endaligned
\end{equation}
for some smooth complex-valued functions~$a_k,\tilde{a}_k,x_k,\tilde{x}_k,f^i_{kl},\tilde{f}^i_{kl}$ on~$R$. Subtracting \eqref{pullbackr} from \eqref{struc2} and using \eqref{pullconn2} gives in components
\begin{equation}\label{subtr}
0=\cdots+\i\left(x_k\zeta^k+\tilde{x}_k\bar{\zeta}^k\right)\wedge\bar{\zeta}^i
\end{equation}
where the unwritten summands are not of the form~$\bar{\zeta}^k\wedge\bar{\zeta}^i$. It follows that \eqref{subtr} can hold for every~$i=1,\ldots,n$ if and only if~$\tilde{x}_k=0$. 
\end{proof}

\begin{corollary}
Locally every~$\beta$-integrable~$2$-Segre structure~$\pi : \B \to M$ can be reduced to a torsion-free~$\H$-structure. 
\end{corollary}

\begin{proof}
For a given point~$p \in M$, choose~$q \in \X$ with~$\rho(q)=p$ and a coordinate neighbourhood~$U_p$. Let~$\mu^i$,~$i=1,\ldots,n$ and~$\kappa$ be a basis for the~$(1,\! 0)$-forms on~$\rho^{-1}(U_p)$ as constructed in Lemma \ref{lemma1}. Using Lemma \ref{compsub} there exists a complex~$2n$-submanifold~$f : \Sigma \to \rho^{-1}(U_p)$ passing through~$q$ for which 
$$
f^*\left(\i\mu^1\wedge\bar{\mu}^1\wedge\cdots\wedge\i\mu^n\wedge\bar{\mu}^n\right)\neq 0.
$$ Since the~$\pi : \B \to M$ pullback of a volume form on~$M$ is a nowhere vanishing multiple of~$Z=\i\mu^1\wedge\bar{\mu}^1\wedge\cdots\wedge\i\mu^n\wedge\bar{\mu}^n$, the~$\rho$ pullback of a volume form on~$U_p$ is a nowhere vanishing multiple of~$Z$ and hence~$\rho \circ f : \Sigma \to U_p$ is a local diffeomorphism. Composing~$f$ with the locally available inverse of this local diffeomorphism one gets a local~$\rho$-section which is defined in a neighbourhood of~$p$ and which is a complex submanifold. Applying Theorem \ref{main} it follows that~$\pi : \B \to M$ locally has an underlying torsion-free~$\H$-structure.     
\end{proof}

\section{The flat case}\label{quad}
In this section we apply the obtained results to the Grassmannian of oriented~$2$-planes in~$\R^{n+2}$ which carries a $2$-oriented torsion-free~$2$-Segre structure together with an adapted connection of vanishing curvature. 

Here a $2$-Segre structure $\pi : \B \to M$ is called $2$-oriented if the structure group $H(2,n)$ has been reduced to $H^+(2,n)=\mathrm{GL}^+(2,\R)\otimes \mathrm{GL}(n,\R)$. 

Using Theorem \ref{upperholom} we also get: If~$\B \to M$ is a~$\beta$-integrable~$2$-oriented~$2$-Segre structure, then~$\rho : \X=\B/(\H) \to M$ together with its canonical almost complex structure~$\AJ$ is a quasiholomorphic fibre bundle with fibre~$\GL^+(2,\R)/\GL(1,\C)\simeq D^2$, the open unit disk in~$\C$.   

\subsection{The Grassmannian of oriented 2-planes}
We consider the projective linear group $\mathrm{PL}(n+2,\R)=\mathrm{GL}(n+2,\R)/Z$ which acts transitively from the left on the Grassmannian~$\GR$ of oriented~$2$-planes in~$\R^{n+2}$. The stabiliser subgroup of any element~$\E \in \GR$ may be identified with the subgroup~$S$ consisting of elements of the form
$$
\left[\begin{array}{cc} a & b \\ 0 & c \end{array}\right]
$$
where~$a \in \GL^+(2,\R), c \in \GL(n,\R)$ and~$b \in M_{\R}(2,n)$ is a real~$(2\times n)$-matrix.
Let~$\mu : \mathrm{PL}(n+2,\R) \to \GR\simeq \mathrm{PL}(n+2,\R)/S$ be the quotient projection and write 
$$
\tilde{\theta}=\left(\begin{array}{cc} \alpha & \beta\\ \eta & \gamma\end{array}\right) 
$$
for the Maurer-Cartan form of~$\mathrm{PL}(n+2,\R)$. The real matrix-valued~$1$-forms~$\alpha,\beta,\gamma,\eta$ have sizes according to the block decomposition of the Lie group~$S$ and satisfy~$\text{Tr}(\alpha)+\text{Tr}(\gamma)=0$. Let~$H^i_k$ denote the vector fields dual to the forms~$\eta$ with respect to the coframing~$\tilde{\theta}$. Let~$N\subset \mathrm{PL}(n+2,\R)$ be the closed normal subgroup given by
$$
N=\left\{\left[\begin{array}{cc} \id_2 & b \\ 0 & \id_n \end{array}\right]\;\bigg\vert\; b \in M_{\R}(2,n)\right\}
$$
whose elements will be denoted by~$[b]$. The quotient Lie group~$S/N$ is isomorphic to~$H^+(2,n)$ and thus~$\mathrm{PL}(n+2,\R)/N$ is the total space of a right principal~$\G$-bundle over~$\GR$. Consider the smooth map 
$$
\varphi^i_k : \mathrm{PL}(n+2,\R) \to T\GR, \quad p \mapsto \mu^{\prime}_p(H^i_k(p))\\
$$    
The Maurer-Cartan equation~$\d\tilde{\theta}+\tilde{\theta}\wedge\tilde{\theta}=0$ implies that the form~$\eta$ is basic for the quotient projection~$\mathrm{PL}(n+2,\R) \to \mathrm{PL}(n+2,\R)/N$. Therefore the maps~$\varphi^i_k$ are invariant under the right action of~$N$ and thus descend to smooth maps~$\mathrm{PL}(n+2,\R)/N \to T\GR$. The images~$\varphi^i_k(p)$ for a given point~$p \in \mathrm{PL}(n+2,\R)$ are linearly independent and thus induce a map~$\varphi$ into the coframe bundle of~$\GR$. The maps~$\varphi^i_k$ can be arranged so that the induced map~$\varphi$ from~$\mathrm{PL}(n+2,\R)/N$ into the coframe bundle of~$\GR$ pulls back the components of the canonical~$\C^n$-valued~$1$-form to~$\i (\eta^k_1+\i\eta^k_2)$. It follows again with the Maurer-Cartan equation that~$\varphi$ embeds~$\mathrm{PL}(n+2,\R)/N$ as a smooth right principal~$\G$-subbundle of the coframe bundle of~$\GR$. This subbundle will be denoted by~$\pi_0 : F_0 \to \GR$ and the projection~$\mathrm{PL}(n+2,\R) \to F_0$ by~$\upsilon$. Write 
$$
\aligned
\tilde{\omega}&=\alpha^2_1,\\
2\,\tilde{\xi}&=\left(\alpha^1_2+\alpha^2_1\right)+\i\left(\alpha^2_2-\alpha^1_1\right),\\
\tilde{\phi}&=\gamma-\id_n\alpha^2_2,\\
\zeta^i&=\i(\eta^i_1+\i\eta^i_2).\\
\endaligned
$$
Then straightforward computations show that the forms~$(\tilde{\omega},\tilde{\xi},\tilde{\phi})$ transform under the right action of~$\G$ as the connection forms of an~$\mathcal{S}_0$-adapted connection do. Moreover we have
$$ 
\d\zeta=-\left(\i\left(\tilde{\omega}-\tilde{\xi}\right)\id_n+\tilde{\phi}\right)\wedge\zeta-\i\tilde{\xi}\id_n \wedge \bar{\zeta}, 
$$
This implies that there exists an adapted torsion-free connection~$(\omega,\xi,\phi)$ on~$F_0$ such that 
\begin{equation}\label{pullconn}
\upsilon^*(\omega,\xi,\phi)=(\tilde{\omega},\tilde{\xi},\tilde{\phi}), \; \text{mod}\; \eta, 
\end{equation}
i.e.~\eqref{pullconn} holds up to linear combinations of the elements of~$\eta$. Furthermore the Maurer-Cartan equation implies that the curvature forms of this connection all vanish. Summarising we have proved:  
\begin{proposition}\label{charcomplexcartan}
The Grassmannian of oriented~$2$-planes in $\R^{n+2}$ admits a~$2$-oriented $2$-Segre structure~$\mathcal{S}_0$ together with an adapted, torsion-free, flat connection $(\omega,\xi,\phi)$ such that
$
\upsilon^*(\omega,\xi,\phi)=(\tilde{\omega},\tilde{\xi},\tilde{\phi}), \; \text{mod}\; \eta, 
$
holds.
\end{proposition}

Let~$\rho : X_0 \to \GR$ be the~$D^2$-bundle associated to the~$2$-oriented torsion-free~$2$-Segre structure~$\pi_0 : F_0 \to \GR$ and~$\AJ_0$ its canonical almost complex structure which makes~$(X_0,\AJ_0)$ into a quasiholomorphic fibre bundle with fibre~$D^2$. Its total space~$X_0$ can be identified with the quotient~$\mathrm{PL}(n+2,\R)/\tilde{P}$ where~$\tilde{P}$ is the closed Lie subgroup 
$$
\tilde{P}=\left\{\left[\begin{array}{cc} a & b \\ 0 & c \end{array}\right]\;\bigg\vert\; a \in \GL(1,\C), b \in M_{\R}(2,n), c \in \GL(n,\R) \right\}.
$$
Write an element~$[g] \in \mathrm{PL}(n+2,\R)$ as~$[g_1,\ldots,g_{n+2}]$ where the elements~$g_k$ are column-vectors well defined up to a common non-zero factor. Consider the smooth map  
$$
\lambda : \mathrm{PL}(n+2,\R) \to \mathbb{CP}^{n+1}\setminus \mathbb{RP}^{n+1},\quad
[g_1,g_2,\ldots,g_{n+2}] \mapsto [g_1+\i g_2].
$$
Clearly~$\lambda$ is a surjective submersion whose fibres are the~$\tilde{P}$-orbits and thus induces a diffeomorphism~$\varphi : X_0 \to \mathbb{CP}^{n+1}\setminus\mathbb{RP}^{n+1}$. Therefore~$\rho_0=\rho \circ \varphi^{-1} : \cip \to \GR$ is a bundle with fibre~$D^2$. Explicitly~$\rho_0$ is given by 
$
[z] \mapsto \R\{\Re(z),\Im(z)\}
$
and the~$2$-plane~$\R\{\Re(z),\Im(z)\}$ is oriented by declaring~$\Re(z),\Im(z)$ to be a positively oriented basis. 

\begin{proposition} There exists a biholomorphic fibre bundle isomorphism $\varphi : (X_0,\AJ_0) \to \mathbb{CP}^{n+1}\setminus\mathbb{RP}^{n+1}$ covering the identity on~$\GR$. 
\end{proposition}
\begin{proof}
Using Lemma \ref{charcomplex} and Proposition \ref{charcomplexcartan} it is sufficient to show that~$\lambda$ pulls-back the $(1,\!0)$-forms of $\cip$ to linear combinations of the forms $\zeta^1,\ldots,\zeta^n, \tilde{\xi}$. 
This is a computation which causes no difficulties and so we omit it. 
\end{proof}
\subsection{Smooth quadrics without real points}

Let $V$ be a real vector space. We denote by~$V_{\C}=V\otimes \C$ its complexification and by~$\mathbb{P}(V_{\C})=\left(V_{\C} \setminus \{0\}\right)/\C^*$ its complex projectivisation. An element~$[z] \in \mathbb{P}(V_{\C})$ for which~$z$ is a simple vector is called a \textit{real point}. 

The aim of this subsection is to show that the smooth quadrics~$Q \subset \mathbb{CP}^{n+1}=\mathbb{P}(\R^{n+2}_\C)$ without real points are in one-to-one correspondence with the sections of the bundle~$\rho_0 : \cip \to \GR$ having holomorphic image. This is done by reducing the problem to the case~$n=1$ which was shown to be true in~\cite[Corollary 2]{mettlerprosur} (see also~\cite[Theorem 9]{MR1466165}).  

Let~$\E \subset \R^{n+2}$ be a~$3$-dimensional linear subspace. Choosing an isomorphism~$\R^3\simeq \E$ induces an embedding of the~$2$-sphere~$S^2\simeq G^+_2(\R^3) \hookrightarrow  \GR$. Clearly the image of this embedding and its induced smooth structure do not depend on the chosen isomorphism and thus~$\E$ determines a smoothly embedded~$2$-sphere in~$\GR$ which will be denoted by~$S_\E$. Moreover the isomorphism~$\R^3\simeq \E$ induces a holomorphic embedding~$\mathbb{CP}^2\simeq \mathbb{P}(\E_{\C})\hookrightarrow \mathbb{CP}^{n+1}$ and thus an embedding~$\mathbb{CP}^2\setminus \mathbb{RP}^2 \hookrightarrow \cip$. Again the image of this embedding and its induced complex structure do not depend on the chosen isomorphism and thus~$\E$ determines a holomorphically embedded submanifold~$Y_\E \subset \cip$.  Restricting the base point projection~$\rho_0 : \cip \to \GR$ to~$Y_\E$ gives a~$D^2$-bundle~$\rho_\E : Y_\E \to S_\E$ which is isomorphic to the bundle~$\rho_0^2 : \mathbb{CP}^2 \setminus \mathbb{RP}^2 \to G^+_2(\R^3)$,~$[z] \mapsto \R\left\{\Re(z),\Im(z)\right\}$.     

Recall that for a smooth algebraic hypersurface~$X \subset \P(\V)$, the Gauss map~$\mathcal{G}_X : X \to G_{n-1}(V_{\C})$ sends a point~$x \in X$ to the tangent hyperplane of~$X$ at~$x$. The dual variety~$X^*$ is now defined to be the image of~$X$ under the Gauss map. Usually~$\mathcal{G}_X$ is assumed to take values in~$\mathbb{P}(V^*_{\C})=\mathbb{P}((V_{\C})^*)\simeq G_{n-1}(V_{\C})$. Note that if~$Q\subset \mathbb{P}(V_{\C})$ a smooth quadric without real points. Then the dual of~$Q$ is a smooth quadric without real points in~$\mathbb{P}(V^*_{\C})$. 
It follows that the intersection of a smooth quadric without real points with a real~$k$-plane $\Pi$ of dimension at least $2$ gives again a smooth quadric without real points in $\mathbb{P}(\E_{\C})$. \begin{theorem}\label{onetoone}
The sections of the bundle~$\rho_0 : \cip \to \GR$ having holomorphic image are in one-to-one correspondence with the smooth quad\-rics $Q\subset \mathbb{CP}^{n+1}$ without real points. 
\end{theorem}

\begin{proof}
Let~$\sigma : \GR \to \cip$ be a~$\rho_0$-section with holomorphic image~$Q=\im\sigma$. Let~$\E \subset \R^{n+2}$ be a~$3$-dimensional linear subspace and~$\iota_\E : S_\E \to \GR$,~$\tilde{\iota}_\E : Y_\E \to \cip$ the corresponding embedded submanifolds. Then the map~$\sigma \circ \iota_\E : S_\E \to \cip$ is smooth and takes values in~$Y_\E$. Consequently, the induced map~$\sigma_\E : S_\E \to Y_\E$ is an injective immersion and thus, since~$S_\E$ is compact, a smooth embedding. Set
$
Q_\E=Q\cap Y_\E=\sigma_\E(S_\E),
$
then~$Q_{\E}\subset Y_\E$ is  a smoothly embedded submanifold. Now Chow's theorem implies that~$Q$ is a smooth algebraic hypersurface. Suppose~$P : \C^{n+2} \to \C$ is a homogeneous polynomial defining~$Q$ and let~$P_\E : \C^3 \to \C$ denote the homogeneous polynomial obtained by pulling back~$P$ to~$\E_{\C}\simeq \C^3$. The map~$P_\E$ is a homogeneous polynomial of the same degree as~$P$ which has no real points, since~$P$ has no real points. Under the identification~$Y_\E \simeq \mathbb{CP}^2\setminus\mathbb{RP}^2$,~$Q_\E$ becomes a smoothly embedded submanifold of~$\mathbb{CP}^2\setminus\mathbb{RP}^2$ defined by the zero locus of the homogeneous polynomial~$P_\E$. Since~$Q_\E$ is diffeomorphic to the~$2$-sphere, the genus of~$Q_\E$ is~$0$ and thus by the degree-genus formula for smooth plane algebraic curves
$
g=(d-1)(d-2)/2,
$ 
the degree of~$P_\E$ must be~$1$ or~$2$. However since~$Q_\E$ has no real points the degree of~$P_\E$ and thus the degree of~$P$ must be~$2$. 

Conversely let~$Q\subset \mathbb{CP}^{n+1}$ be a smooth quadric without real points. Let $\left\{\E^{\iota}\right\}_{\iota \in I}$ be a family of~$3$-dimensional linear subspaces of~$\R^{n+2}$ so that the submanifolds~$S_{\E^{\iota}}$ cover~$\GR$.  Let~$Q_{\E^{\iota}}$ denote the intersection of~$Q$ with~$\mathbb{P}(\E^{\iota}_{\C})$ which is a smooth quadric without real points.  According to~\cite[Corollary 2]{mettlerprosur} each such quadric is the image of a unique section~$\sigma_{\iota} : S_{\E^{\iota}} \to X_{\E^{\iota}}$. Now for any two~$\E^{\iota_1}, \E^{\iota_2}$ the spheres~$S_{\E^{\iota_1}}$ and~$S_{\E^{\iota_2}}$ are either disjoint or intersect in exactly two points. Since for a given~$Q_{\E^{\iota}}$ the section~$\sigma_{\iota}$ is unique, it follows that~$Q_{\E^{\iota}}$ intersects each~$\rho_{\E^{\iota}}$-fibre in exactly one point. This implies that the sections~$\sigma_{\iota_1}$ and~$\sigma_{\iota_2}$ agree on intersection points and thus the family~$\left\{\sigma_{\iota}\right\}_{\iota \in I}$ gives rise to a unique global section~$\sigma : \GR \to \cip$ with image~$Q$. 
\end{proof}
\begin{corollary}\label{karin}
The torsion-free~$\H$-reductions~$R \subset F_0$ are in one-to-one correspondence with the smooth quadrics~$Q\subset \mathbb{CP}^{n+1}$ without real points.  
\end{corollary}

\begin{proof}
This follows immediately from Theorems \ref{main} and \ref{onetoone}.
\end{proof}

\begin{remark} For~$n=2$, the case of conformal~$4$-manifolds of split-signature, Corollary \ref{karin} can also be deduced by applying results from~\cite{MR2286630}. One could also look for~$S^1\cdot \GL(2,\R)$-reductions whose associated almost complex structure is not only integrable, but for which the corresponding conformal structure~$[g]$ also contains a K\"ahler-metric. This, and the related problem in~$(4,\! 0)$-signature has been studied in~\cite{MR2609304} (see also \cite[Theorem D]{MR2286630}). Moreover for~$n=2$, Theorem \ref{main} has an analogue in~$(4,\! 0)$-signature due to Salamon~\cite{MR829230}. 
\end{remark}

\appendix

\section{}

\subsection{The structure group}

We provide a proof for the existence of the isomorphism claimed in \eqref{strucgroup}. 

\begin{lemma}\label{decomp}
Let $g \in G(m,n)$ and $v \in V$. Then precisely one of the two statements holds: 
\begin{itemize}
\item[(i)] There exists $v_0 \in V$ and $b_{v} \in \mathrm{Isom}(W,W)$ such that for all $w \in W$
$$
g(v\otimes w)=v_0 \otimes b_{v}(w). 
$$
\item[(ii)] There exists $w_0\in W$ and $a_{v} \in \mathrm{Isom}(W,V)$ such that for all $w\in W$. 
$$
g(v\otimes w)=a_v(w)\otimes w_0. 
$$ 
\end{itemize}
Moreover if (i) (or (ii)) is true for some $v \in V$, then for all $v \in V$.  
\end{lemma}
\begin{proof}
For $v=0$ the statement is obvious so let's assume $v \neq 0$. Let $w_1,w_2\in W$ be linearly independent, write $g(v\otimes w_1)=v_1\otimes u_1$ and $g(v\otimes w_2)=v_2\otimes u_2$ for some vectors $v_1,v_2 \in V$ and $u_1,u_2 \in W$ all nonzero. Then $g \in G(m,n)$ implies that one of the two following cases occurs:
\begin{itemize}
\item[(I)] $v_1\wedge v_2=0$ and $u_1\wedge u_2\neq 0,$ 
\item[(II)] $v_1\wedge v_2\neq 0$ and $u_1\wedge u_2=0$.
\end{itemize}
Assume (I) holds and fix $v_0\neq 0$ with $v_0\wedge v_1=0$. It follows again with $g \in G(m,n)$, that for every $w \in W$, there exists a unique element $b_v(w)\in W$ such that $g(v\otimes w)=v_0\otimes b_v(w)$. The assignment $w \mapsto b_v(w)$ is invertible and linear, thus (i)  follows. Assuming (II) holds we conclude similarly that (ii) follows. 

Suppose case (i) occurs for some $v \in V$ and case (ii) for some $v^{\prime} \in V$. Let $b_v$ and $a_{v^{\prime}}$ be the associated isomorphisms. Since $g \in G(m,n)$ either $b_v$ or $a_{v^{\prime}}$ must have rank $1$, thus contradicting the fact that both maps are isomorphisms.  
\end{proof}
We can now show:
\begin{proposition}
We have an isomorphism
$$
G(m,n)\simeq\left\{\begin{array}{cc} H(m,n) & n \neq m,\\ H(n,n)\rtimes \mathbb{Z}_2  & n=m.\\ \end{array}\right.
$$
\end{proposition}

\begin{proof} Let $g \in G(m,n)$. Assume case (i) of Lemma \ref{decomp} holds for some and hence all $v \in V$. Let $\hat v,\tilde{v} \in V$, then by Lemma \ref{decomp} there exists $\hat v_0,\tilde{v}_0 \in V$ and $b_{\hat{v}},b_{\tilde{v}} \in \mathrm{Isom}(W,W)$ such that for all $w \in W$ we have
$$
g(\hat{v}\otimes w)=\hat{v}_0\otimes b_{\hat{v}}(w), \quad \varphi(\tilde{v}\otimes w)=\tilde{v}_0\otimes b_{\tilde{v}}(w). 
$$
On the other hand for some $\tilde{w} \in W$ there must exist $\tilde{w}_0 \in W$ and $a_{\tilde{w}} \in \mathrm{Isom}(V,V)$ such that for all $v \in V$ 
$$
g(v\otimes \tilde{w})=a_{\tilde{w}}(v)\otimes \tilde{w}_0. 
$$
We thus get
$$
g(\hat{v}\otimes \tilde{w})=a_{\tilde{w}}(\hat{v})\otimes \tilde{w}_0=\hat{v}_0\otimes b_{\hat{v}}(\tilde{w})
$$
and
$$
g(\tilde{v}\otimes \tilde{w})=a_{\tilde{w}}(\tilde{v})\otimes \tilde{w}_0=\tilde{v}_0\otimes b_{\tilde{v}}(\tilde{w})
$$
Since this holds for any $\tilde{w} \in W$, the map $b_{\tilde{v}}$ must be a (nonzero) constant multiple of the map $b_{\hat{v}}$. It follows that there exists $a \in \mathrm{Isom}(V,V)$ and $b \in \mathrm{Isom}(W,W)$ such that $g(v\otimes w)=a(v)\otimes b(w)$ for all $v\in V$ and $w \in W$. If the case (ii) of Lemma \ref{decomp} holds, we can conclude similarly that there exists $a \in \mathrm{Isom}(W,V)$ and $b \in \mathrm{Isom}(V,W)$ and  such that $g(v\otimes w)=a(w)\otimes b(v)$ for all $v \in V$ and $w \in W$. From this the claim follows easily. 
\end{proof}

\subsection{Segre and almost Grassmann structures}
Finally, we show that for $n+m$ odd, an $(m,n)$-Segre structure $\pi : \B \to M$ is the same as an almost Grassmann structure. Let $S(m,n)$ denote the subgroup of $\mathrm{GL}(m,\R)\times \mathrm{GL}(n,\R)$ consisting of pairs $(a_m,a_n)$ satisfying $\det a_m\det a_n=1$. Clearly for $n+m$ odd, the map
$$
\rho : S(m,n) \to \mathrm{GL}(m,\R)\otimes \mathrm{GL}(n,\R), \quad (a_m,a_n) \mapsto a_m\otimes a_n
$$
is a Lie group isomorphism. For $k=m,n$ let $\chi_k : S(m,n) \to \mathrm{Aut}(\R^k)$ be the representation defined by 
$$
\chi_k\left((a_m,a_n)\right)(v)=a_kv 
$$
for $v \in \R^k$. The reader will recall that any (real or complex) representation $\chi : H(m,n)\simeq S(m,n) \to \mathrm{Aut}(V)$ defines a vector bundle $(\B)_{\chi}=\B\times_{\chi} V$ over $M$. Let $E_k\to M$ denote the rank $k$ vector bundle obtained via the representation $\chi_k$. By construction, the vector bundle associated to the representation $\chi_m\otimes \chi_n$ is the tangent bundle of $M$ and we thus obtain an isomorphism 
$$
TM \simeq E_m\otimes E_n
$$
inducing the Segre structure $\pi : \B \to M$. 

\subsection*{Acknowledgements}
The author is grateful to Robert L. Bryant for helpful discussions. The author also would like to thank the referees for pointing out a mistake in an earlier version of this article and for other valuable suggestions.  

\providecommand{\bysame}{\leavevmode\hbox to3em{\hrulefill}\thinspace}
\providecommand{\noopsort}[1]{}
\providecommand{\mr}[1]{\href{http://www.ams.org/mathscinet-getitem?mr=#1}{MR~#1}}
\providecommand{\zbl}[1]{\href{http://www.zentralblatt-math.org/zmath/en/search/?q=an:#1}{Zbl~#1}}
\providecommand{\jfm}[1]{\href{http://www.emis.de/cgi-bin/JFM-item?#1}{JFM~#1}}
\providecommand{\arxiv}[1]{\href{http://www.arxiv.org/abs/#1}{arXiv~#1}}
\providecommand{\doi}[1]{\href{http://dx.doi.org/#1}{DOI}}
\providecommand{\MR}{\relax\ifhmode\unskip\space\fi MR }
\providecommand{\MRhref}[2]{%
  \href{http://www.ams.org/mathscinet-getitem?mr=#1}{#2}
}
\providecommand{\href}[2]{#2}


\begin{thebibliography}{10}

\bibitem{MR1406793}
\bgroup\scshape{}M.~A. Akivis\egroup{} and \bgroup\scshape{}V.~V.
  Goldberg\egroup{}, \emph{Conformal differential geometry and its
  generalizations}, \emph{Pure and Applied Mathematics (New York)}, John Wiley
  \& Sons Inc., New York, 1996, A Wiley-Interscience Publication.
  \mr{1406793}\;  \zbl{0863.53002}\;

\bibitem{MR1692442}
\bysame, Semiintegrable almost {G}rassmann structures,  \emph{Differential
  Geom. Appl.} \textbf{10} (1999), 257--294. \mr{1692442}\;  \zbl{0921.53006}\;

\bibitem{MR1085595}
\bgroup\scshape{}T.~N. Bailey\egroup{} and \bgroup\scshape{}M.~G.
  Eastwood\egroup{}, Complex paraconformal manifolds---their differential
  geometry and twistor theory,  \emph{Forum Math.} \textbf{3} (1991), 61--103.
  \mr{1085595}\;  \zbl{0728.53005}\;

\bibitem{MR1106939}
\bgroup\scshape{}R.~J. Baston\egroup{}, Almost {H}ermitian symmetric manifolds.
  {I}. {L}ocal twistor theory,  \emph{Duke Math. J.} \textbf{63} (1991),
  81--112. \mr{1106939}\;  \zbl{0724.53019}\;

\bibitem{MR1083148}
\bgroup\scshape{}R.~L. Bryant\egroup{}, \bgroup\scshape{}S.~S. Chern\egroup{},
  \bgroup\scshape{}R.~B. Gardner\egroup{}, \bgroup\scshape{}H.~L.
  Goldschmidt\egroup{}, and \bgroup\scshape{}P.~A. Griffiths\egroup{},
  \emph{Exterior differential systems}, \emph{Mathematical Sciences Research
  Institute Publications} \textbf{18}, Springer-Verlag, New York, 1991.
  \mr{1083148}\;  \zbl{0726.58002}\;

\bibitem{MR1427757}
\bgroup\scshape{}R.~L. Bryant\egroup{}, Classical, exceptional, and exotic
  holonomies: a status report,  in \emph{Actes de la {T}able {R}onde de
  {G}\'eom\'etrie {D}iff\'erentielle ({L}uminy, 1992)}, \emph{S\'emin. Congr.}
  \textbf{1}, Soc. Math. France, Paris, 1996, pp.~93--165. \mr{1427757}\;
  \zbl{0882.53014}\;

\bibitem{MR1466165}
\bysame, Projectively flat {F}insler {$2$}-spheres of constant curvature,
  \emph{Selecta Math. (N.S.)} \textbf{3} (1997), 161--203. \mr{1466165}\;
  \zbl{0897.53052}\;

\bibitem{MR1898190}
\bysame, Some remarks on {F}insler manifolds with constant flag curvature,
  \emph{Houston J. Math.} \textbf{28} (2002), 221--262, Special issue for S. S.
  Chern. \mr{1898190}\;  \zbl{1027.53086}\;

\bibitem{MR2389257}
\bgroup\scshape{}M.~Crampin\egroup{} and \bgroup\scshape{}D.~J.
  Saunders\egroup{}, Path geometries and almost {G}rassmann structures,  in
  \emph{Finsler geometry, {S}apporo 2005---in memory of {M}akoto {M}atsumoto},
  \emph{Adv. Stud. Pure Math.} \textbf{48}, Math. Soc. Japan, Tokyo, 2007,
  pp.~225--261. \mr{2389257}\;  \zbl{1168.53010}\;

\bibitem{MR2609304}
\bgroup\scshape{}M.~Dunajski\egroup{} and \bgroup\scshape{}P.~Tod\egroup{},
  Four-dimensional metrics conformal to {K}\"ahler,  \emph{Math. Proc.
  Cambridge Philos. Soc.} \textbf{148} (2010), 485--503. \mr{2609304}\;
  \zbl{1188.53078}\;

\bibitem{MR925263}
\bgroup\scshape{}A.~B. Goncharov\egroup{}, Generalized conformal structures on
  manifolds,  \emph{Selecta Math. Soviet.} \textbf{6} (1987), 307--340,
  Selected translations. \mr{925263}\;  \zbl{0632.53038}\;

\bibitem{MR1847382}
\bgroup\scshape{}D.~A. Grossman\egroup{}, Torsion-free path geometries and
  integrable second order {ODE} systems,  \emph{Selecta Math. (N.S.)}
  \textbf{6} (2000), 399--442. \mr{1847382}\;  \zbl{0997.53013}\;

\bibitem{MR0200858}
\bgroup\scshape{}T.~Hangan\egroup{}, G\'eom\'etrie diff\'erentielle
  grassmannienne,  \emph{Rev. Roumaine Math. Pures Appl.} \textbf{11} (1966),
  519--531. \mr{0200858}\;  \zbl{0163.43402}\;

\bibitem{MR0291976}
\bgroup\scshape{}T.~Ishihara\egroup{}, On tensor-product structures and
  {G}rassmannian structures,  \emph{J. Math. Tokushima Univ.} \textbf{4}
  (1970), 1--17. \mr{0291976}\;

\bibitem{MR2286630}
\bgroup\scshape{}C.~Lebrun\egroup{} and \bgroup\scshape{}L.~J. Mason\egroup{},
  Nonlinear gravitons, null geodesics, and holomorphic disks,  \emph{Duke Math.
  J.} \textbf{136} (2007), 205--273. \mr{2286630}\;  \zbl{1113.53032}\;

\bibitem{MR1804138}
\bgroup\scshape{}Y.~Machida\egroup{} and \bgroup\scshape{}H.~Sato\egroup{},
  Twistor theory of manifolds with {G}rassmannian structures,  \emph{Nagoya
  Math. J.} \textbf{160} (2000), 17--102. \mr{1804138}\;  \zbl{1039.53055}\;

\bibitem{mettlerprosur}
\bgroup\scshape{}T.~Mettler\egroup{}, \emph{{W}eyl metrisability for projective
  surfaces}, preprint,
  \href{http://arxiv.org/abs/0910.2618}{ar{X}iv:0910.2618v2} [math.DG], 2009.

\bibitem{MR812312}
\bgroup\scshape{}N.~R. O'Brian\egroup{} and \bgroup\scshape{}J.~H.
  Rawnsley\egroup{}, Twistor spaces,  \emph{Ann. Global Anal. Geom.} \textbf{3}
  (1985), 29--58. \mr{812312}\;  \zbl{0526.53057}\;

\bibitem{MR829230}
\bgroup\scshape{}S.~Salamon\egroup{}, Harmonic and holomorphic maps,  in
  \emph{Geometry seminar ``{L}uigi {B}ianchi'' {II}---1984}, \emph{Lecture
  Notes in Math.} \textbf{1164}, Springer, Berlin, 1985, pp.~161--224.
  \mr{829230}\;  \zbl{0591.53031}\;

\end{thebibliography}
\end{document}